\documentclass[12pt,reqno]{amsart}
\usepackage[colorlinks=true, pdfstartview=FitV, linkcolor=blue, citecolor=blue, urlcolor=blue]{hyperref}

\usepackage{amscd, amssymb,latexsym,amsmath,  amsmath}
\usepackage[all]{xy}
\usepackage{anysize}
\usepackage[sc]{mathpazo} 
\usepackage{color}
\usepackage{hyperref}

\hypersetup{
    colorlinks=true,
    linkcolor=blue,
    filecolor=magenta,      
    urlcolor=cyan,
    pdftitle={Overleaf Example},
    pdfpagemode=FullScreen,
    }
   
\marginsize{2.5cm}{2.5cm}{2.5cm}{2.5cm}
\newtheorem{theorem}{Theorem}[section] %(If you want theorem numbered
\newtheorem{lemma}[theorem]{Lemma}%               with section number.  Same
\theoremstyle{definition}

\newtheorem{example}{Example}
\newtheorem{definition}{Definition}
\theoremstyle{remark}
\newtheorem{remark}{Remark}

\newcommand{\ep}{\epsilon}
\newcommand{\mbb}{\mathbb}
\newcommand{\de}{\delta}
\newcommand{\ov}{\overline}

\newcommand{\pa}{\partial}
\newcommand{\mf}{\mathbb}
\newcommand{\om}{\omega}
\newcommand{\Om}{\Omega}

\newcommand{\z}{\zeta}

\newcommand{\ti}{\tilde}

\renewcommand{\Re}{\operatorname{Re}}

\begin{document}

\title[Geodesics]{On the Geodesics of the Szeg\"o metric}

\author[Anjali Bhatnagar]{Anjali Bhatnagar}
\address{
Department of Mathematics, Indian Institute of Science Education and Research, Pune 411008, India}
\email{anjali.bhatnagar@students.iiserpune.ac.in}

\begin{abstract}
We explore the existence of closed geodesics and geodesic spirals for the Szeg\"o metric in a $C^\infty$-smoothly bounded strongly pseudoconvex domain $\Om\subset\mf{C}^n$, which is not simply connected for $n \geq 2$.
\end{abstract}

\date{}

\maketitle

%------------------------------------------------------
  
\section{Introduction}
 The purpose of the article is to identify the shared properties by the Szeg\"o and Bergman metrics, which continues our prior work \cite{bb24}. The Bergman metric is an active area of research while the Szeg\"o metric with respect to the Fefferman surface area measure was recently introduced by Barrett-Lee \cite{bl14} to study an invariant version of the Szeg\"o metric. This metric is called the Fefferman-Szeg\"o metric---which has been further investigated in \cite{k19,k21}. In contrast, the Szeg\"o metric relative to the Euclidean surface area measure is generally not invariant under biholomorphisms, except in one dimension. We have explored the intrinsic properties of the Szeg\"o metric such as geodesics, curvature, and $L^2$-cohomology in non-degenerate finitely connected planar domains in \cite{bb24} and draw the similarity between the theories of the Bergman and Szeg\"o metrics. We have also provided the comparison between the Carath\'eodory and Szeg\"o metrics, and established the existence of domains where the curvatures of the Szeg\"o metric achieves both positive and negative real-values. It can also be observed from \cite{bb24,z10} that there are domains in which the curvatures of the Bergman and Szeg\"o metric have opposite signs. 
 
 To deepen this comparison, we continue to study geodesics that remain in a compact subset of domains. Such geodesics are either closed or non-closed---the latter is called \emph{geodesic spiral}. This was explored by Herbort for the Bergman metric in \cite{hg83}. To set the stage, we briefly recall the setup. Let $\Om\subset\mf{C}^n$ be a bounded domain with $C^{\infty}$-smooth boundary $\pa \Om$. The Hardy space $H^{2}(\pa\Om)$ is defined as the closure in $L^{2}(\pa \Om)$ of the set of functions whose Poisson integrals are holomorphic in $\Om$. The Szeg\"o kernel $S_{\Om}(z, w)$ associated with $H^2(\pa\Om)$ is uniquely determined by the following properties: for each $z\in \Om, S_{\Om}(\cdot, z)\in H^{2}(\pa\Om)$, for all $z, w\in \Om, S_{\Om}(z, w)=\ov{S_{\Om}(w, z)}$, and for each $h\in H^{2}(\pa\Om),$
\[h(z)=\int_{\pa\Om}h(w)S_{\Om}(z, w)d\sigma_{E}\;\;\;\text{ for all }z\in \Om,\]
where $d\sigma_{E}$ denotes the Euclidean surface area measure.

\noindent Furthermore, $S_{\Om}(z, w)$ can be expressed in terms of any complete orthonormal basis $\{\phi_{i}\}_{i\geq 1}$ of $H^{2}(\pa \Om)$ as follows
\[S_{\Om}(z,w)=\sum_{i=1}^{\infty}\phi_{i}(z)\ov{\phi_{i}(w)},\]
where the series converges uniformly on compact subsets of $\Om\times \Om$. Consequently, the function $g_\Om(z)=\log S_\Om(z, z)$ is a $C^\infty$-smooth strongly plurisubharmonic function and therefore induces the K\"ahler metric on $\Om$ called the Szeg\"o metric defined as
 \begin{equation*}
      ds_{s_{\Om}}^2=\sum_{j, k=1}^{n}\frac{\partial^{2}g_{\Om}(z)}{\partial z_{j}\partial\overline{z}_{k}}dz_{j}d\ov{z}_{k}.
   \end{equation*}
   \begin{example}
       Let $\Om=\mf{B}^n$ be the unit ball in $\mf{C}^n$. Recall that
       \[S_{\mf{B}^n}(z, w)=\frac{(n-1)!}{2\pi^n}\frac{1}{(1-z\cdot \ov{w})^{n}},\]
       where $z\cdot\ov{w}=\sum_{\ell=1}^nz_\ell\ov{w}_\ell$, and hence \[\frac{\partial^{2}g_{\mf{B}^n}(z)}{\partial z_{j}\partial\overline{z}_{k}}=n\frac{\partial^{2}}{\partial z_{j}\partial\overline{z}_{k}}\log\frac{1}{1-|z|^2}=n\left(\frac{\de_{j\ov{k}}(1-|z|^2)+\ov{z_j}z_k}{(1-|z|^2)^2}\right).\]
      
   \end{example}
   \begin{example}(\cite{bb24})\label{Anuulus}
       For $r\in (0, 1),$ let $\Om=A_r=\{z\in\mf{C}: r<|z|<1\}$ denotes the annulus. It is known that
        \[S_{A_{r}}(z, w)=\frac{1}{2\pi}\sum_{n=-\infty}^{\infty}\frac{(z\overline{w})^{n}}{1+r^{2n+1}}.\]
        Now, let $\wp$ be the Weierstrass elliptic $\wp$-function with periods $2\om_1=-2\log r$ and $2\om_3=2i\pi$. Then, the Szeg\"o metric on $A_r$ can be expressed as
       \[
ds_{s_{A_{r}}}^{2}=\frac{\wp\big(2\log \vert z \vert\big)-\wp\big(2\log \vert z \vert+\om_1+\om_3\big) }{\vert z \vert^2}|dz|^{2}.
\]
   \end{example}
We now recall some definitions to state our main result.
 \begin{definition}\label{spiral-geodesic}
    Let $(X, \ti{g})$ be a complete Riemannian manifold.
    \begin{itemize}
        \item [(a)] A geodesic $c:\mathbb{R}\to X$ is called a \textit{geodesic spiral}, if $c$ is non-closed and there exists a compact subset $K$ of $X$ such that $c(t)\in K$ for all $t\geq 0$.
          \item [(b)] Let $c:\mathbb{R}\to X$ be a nontrivial geodesic, and let $x_0\in X$. If there exist $t_{1},~t_{2}\in\mathbb{R}\text{ with } t_1<t_2$ such that $c(t_{1})=c(t_{2})=x_0$, then the segment $c_{|_{[t_{1},t_{2}]}}$ is called a geodesic loop passing through $x_0$.
    \end{itemize}
\end{definition}
 \begin{theorem}\label{main}
 Let $\Om\subset\mbb {C}^{n}$ be a $C^\infty$-smoothly bounded strongly pseudoconvex domain which is not simply connected, equipped with the Szeg\"o metric $ds_{s_\Om}^2$. We have
 \begin{itemize}
     \item [(a)] Every nontrivial homotopy class of loops in $\Om$ consists a closed geodesic. 
     \item[(b)] Suppose the universal cover of $\Om$ is infinitely sheeted. Then, for each point $z_0\in \Om$ that does not lie on a closed geodesic, there exists a geodesic spiral passing through $z_0$.
 \end{itemize}
 \end{theorem}

  \textbf{Acknowledgements.} The author expresses gratitude to D. Borah for suggesting the problem and to D. Kar for his valuable suggestions.

\section{Closed geodesics}\label{closed-geodesics}
In this section, we prove Theorem \ref{main} (a). We first recall a general result of Herbort.
\begin{theorem}\cite[Theorem 1.1]{hg83}\label{clo-herb}
Let $G\subset\mf{R}^{N}$ be a bounded domain that is not simply connected such that the following holds: 
\begin{itemize}
\item [(i)] For each $p \in \ov G$, there is an open neighbourhood $U\subset\mf{R}^N$ of $p$ such that $G \cap U$ is simply connected.
\item [(ii)] The domain $G$ is equipped with a complete Riemannian metric $\ti{g}$ which has the following property: (B) Given $S>0$, there exists $\de>0$, such that for each $p \in G$ with $d(p, \pa G)<\de$ and every $X \in \mf{R}^N$, $\ti{g}(p, X) \geq S\|X\|^2$, where $\|\cdot \|$ denotes the Euclidean norm. 
\end{itemize}
Then, every nontrivial homotopy class of loops in $G$ contains a closed geodesic for $\ti{g}$.
\end{theorem}
Thus, it is crucial to establish the completeness of the Szeg\"o metric, which is provided by the following result.
\begin{lemma}\label{sgo-complete}
    Let $\Om\subset\mbb {C}^{n}$ be a $C^\infty$-smoothly bounded strongly pseudoconvex domain. Then, the Szeg\"o metric $ds_{s_\Om}$ is complete.
\end{lemma}
This lemma follows from the fact that the Carath\'eodory metric is complete on $C^\infty$-smoothly bounded strongly pseudoconvex domains. Since the Szeg\"o metric dominates the Carath\'eodory metric, which can be shown using similar lines of reasoning as for the Bergman metric---see \cite[Theorem 12.8.1]{jp93}. To see this, we start by recalling the Carath\'eodory metric $ds_{c_\Om}$ on a bounded domain $\Om$: Let $z_0\in\Om, \z\in\mf{C}^n$, 
\[ds_{c_\Om}(z_0,\z)=\sup\left\{\left(\sum_{j=1}^n\Bigg|\frac{\pa\phi(z_0)}{\pa z_j}\z_j\Bigg|^2\right)^{\frac{1}{2}}:\phi:\Om\to\mf{D}\text{ holomorphic and }\phi(z_0)=0\right\}.\]
\begin{lemma}
    For all $z\in\Om,~\z\in\mf{C}^n$, we have 
    \[ds_{s_\Om}(z,\z)\geq ds_{c_\Om}(z,\z).\]
\end{lemma}
\begin{proof}
  First, we express the Szeg\"o metric in terms of maximal domain functions $J_\Om^{(j)},~j=0,1$, which are defined by
 \begin{align*}
     J_\Om^{(0)}(z, \z)&=\sup_{f\in H^{2}(\pa\Om)}\Big\{\big \vert f (z)\big \vert^{2}:  \|f\|_{H^{2}(\pa\Om)}\leq 1\Big\},\text{ and}\\
      J_\Om^{(1)}(z,\z)&=\sup_{f\in H^{2}(\pa\Om)}\left\{\sum_{j=1}^n\Bigg|\frac{\pa f(z)}{\pa z_j}\z_j\Bigg|^2: f(z)=0,\; \|f\|_{H^{2}(\pa\Om)}\leq 1\right\},
 \end{align*}
for $z\in\Om,~\z\in\mf C^n$. Fix $z_0\in\Om,~ \z\in\mf{C}^n$, there exists an orthonormal basis $\{\phi_{k}\}_{k\geq 0}$ of $H^2(\pa\Om)$ such that $\text{for all }k\geq 2$,
\begin{equation}\label{o-n-b}
   \phi_0(z_0)\neq 0,\;\;\; \phi_{k-1}(z_0)=0,\;\;\;\sum_{j=1}^n\frac{\pa\phi_k(z_0)}{\pa z_j}\z_j=0.   
\end{equation}
Using (\ref{o-n-b}), it can shown that 
  \[ds_{s_\Om}(z_0,\z)^2=\frac{J_\Om^{(1)}(z_0,\z)}{J_\Om^{(0)}(z_0,\z)},\;\;\text{ and }\;\;J_\Om^{(0)}(z_0, \z)=S_\Om(z_0, z_0).\]

  \noindent Now, we define \[f(z)=\frac{S_\Om(z, z_0)}{\sqrt{S_\Om(z_0,z_0)}}\phi(z),\]
  where $\phi:\Om\to\mf{D}$ is an arbitrary holomorphic function with $\phi(z_0)=0$. It is evident that $\|f\|_{H^2(\pa\Om)}\leq 1$ and $f(z_0)=0$. Therefore, it can be concluded that \[ds_{s_\Om}(z_0,\z)^2\geq ds_{c_\Om}(z_0,\z)^2.\]
\end{proof}
\begin{proof}[\textbf{Proof of Theorem \ref{main} (a)}]
    It is enough to verify the hypothesis of the Theorem \ref{clo-herb} for $G=\Om$ and $\ti{g}=ds_{s_{\Om}}^2$. Clearly, Condition (i) is satisfied by the smoothness of $\pa\Om$. For Condition (ii), observe that the completeness of the Szeg\"o metric follows from the preceding lemma. Therefore, it remains to verify Property (B), which can be deduced from \cite[Lemma 1]{h73}. Thus, the proof is complete.
\end{proof}

\begin{remark}
 The unit ball $\mf B^n$ has no nontrivial closed geodesic for the Szeg\"o metric because $(\mf B^n, ds_{s_{\mf {B}^n}}^2)$ is a Hadamard-Cartan manifold.   
\end{remark}

%In this article, we explore the intrinsic properties of the Szeg\"o metric, such as geodesics and $L^2$-cohomology and also compare the Szeg\"o metric with the classical metrics, including the Carath\'eodory, Kobayashi, and Bergman metrics. The common tool in the proof of these results is the boundary behaviour of the Szeg\"o metric, which does not rely on any invariance property. Therefore, we refer to the classical Szeg\"o metric in our discussions. However, one can check that all the results presented here will hold for the invariant Szeg\"o metric using similar lines of reasoning.

%In this article, our focus is on the classical Szeg\"o metric, as invariance is not required for our purposes. Our main objective is to explore the intrinsic properties of the Szeg\"o metric, such as geodesics and $L^2$-cohomology. 

\section{Geodesic spirals}\label{geodesic-spirals}
Our next goal is to prove Theorem \ref{main}~(b). We start by recalling a result of Herbort.
\begin{lemma}(\cite{hg83})\label{Lem-Her-cpt}
    Let $(X, \ti{g})$ be a complete Riemannian manifold whose universal cover is infinitely sheeted, and let $x_0\in X$ such that there are no closed geodesic passes. If there exists a compact subset $K$ of $X$ with the property that each geodesic loop passing through $x_0$ lies within $K$, then there exists a geodesic spiral for $\ti g$ passing through $x_0$.
\end{lemma}
From Lemma \ref{sgo-complete}, the Szeg\"o metric on smoothly bounded strongly pseudoconvex domains is complete. Therefore, the existence of a geodesic spiral for $ds_{s_\Om}^2$ reduces to identifying an appropriate compact subset $K$ of $\Om$---which is addressed by the following theorem.
\begin{theorem}\label{geo-spi-l1}
    Let $\Om=\{\rho<0\}\subset\mf{C}^n\text{ where }n\geq 2,$ be a $C^\infty$-smoothly bounded strongly pseudoconvex domain with a $C^\infty$-smooth strongly plurisubharmonic defining function $\rho$. Then there exists $\ep=\ep(\Om)>0$ such that for every geodesic $c:\mf{R}\to\Om$ of the Szeg\"o metric $ds_{s_{\Om}}^2$ satisfying $(\rho\circ c)(0)>-\ep$ and $(\rho\circ c)'(0)=0$, it follows that $(\rho\circ c)'(0)>0$.
\end{theorem}
Before giving a proof of Theorem \ref{geo-spi-l1}, let us complete the proof of the main theorem.
\begin{proof}[\textbf{Proof of Theorem \ref{main} (b)}]
    Let $z_{0}\in \Om$ be a point through which no closed geodesic passes and let $\rho$ and $\ep$ be as in Theorem~\ref{geo-spi-l1}. Now, we define
\[
\epsilon_{1}=\min\{\epsilon,~-\rho(z_{0})\}\;\text{ and }\;K=\big\{z\in \Om: \rho(z)\leq -\epsilon_{1} \big\}.
\]
It can be seen that the compact set $K$ has the desired property as stated in Lemma~\ref{Lem-Her-cpt}. Indeed, let $c |_{[t_{1},t_{2}]}: [t_{1},t_{2}]\to \Om$ be a geodesic loop that passes through $z_{0}$ and assume that $c |_{[t_{1},t_{2}]}([t_{1},t_{2}])\not\subset K$. Since $(\rho\circ c)|_{[t_{1},t_{2}]}$ is a continuous real-valued function, it achieves maximum at some point $t_{0}\in (t_{1}, t_{2})$. Thus, by the definition of $K$, it follows that 
\[(\rho\circ c)(t_{0})>-\epsilon,~(\rho\circ c)'(t_{0})=0\text{ and }(\rho\circ c)''(t_{0}) \leq 0.\]
This, however, contradicts Theorem \ref{geo-spi-l1}. Therefore, by Lemma~\ref{Lem-Her-cpt}, the proof of (ii) is established.
\end{proof}
Finally, to present a proof of Theorem \ref{geo-spi-l1}, we recall one of the most elegant results in complex analysis: the asymptotic expansion of the Szeg\"o kernel, given by Fefferman \cite{fc74} and  Boutet de Monvel-Sj\"ostrand \cite{bs76}.
\begin{theorem}(Fefferman \cite{fc74}, Boutet de Monvel-Sj\"ostrand \cite{bs76})   Let $\Om=\{\rho<0\}\subset\mf{C}^n$ be a $C^\infty$-smoothly bounded strongly pseudoconvex domain. Then, there exist functions $\Phi, \Psi\in C^\infty(\ov\Om)$ with $\Phi(z)>0$ near $\pa\Om$ such that the diagonal values of the Szeg\"o kernel satisfies
    \begin{equation}\label{F}
    S_{\Om}(z, z)=\frac{h(z)}{|\rho(z)|^{n}}=\frac{\Phi(z)+\Psi(z)|\rho(z)|^{n}\log |\rho(z)|}{|\rho(z)|^{n}}.
\end{equation}
 
\end{theorem}
We now introduce some notations. When the domain $\Om$ remains fixed throughout the proof, we omit $\Om$ from the notation. For instance, we write $S_{\Om}(z, z)$ as $S(z, z)$, and so forth. Let $f$ be a $C^\infty$-smooth real-valued. For $j,k,l=1,\ldots, n$, define 
\[f_{j}=\frac{\pa f}{\pa z_j},\; f_{\ov{j}}=\ov{f_{j}},\; f_{jk}:=\frac{\pa f_{j}}{\pa z_k},\; f_{jk\ov{l}}=\frac{\pa f_{jk}}{\pa\ov{z}_l},\; f_{j\ov{k}}=\frac{\pa f_{j}}{\pa \ov{z}_{k}}\text{ and so on}.\]
Let $L_{f}(z)=\Big(f_{j\ov{k}}(z)\Big)_{j,k=1}^n$ denotes the Levi matrix. If $L_{f}(z)$ is positive definite, then ${f}^{j\overline{k}}(z)$ represent the coefficients of its inverse $L_{f}(z)^{-1}$. Moreover, we set $\mathfrak{g}(z)=-\log\big|\rho(z)\big|$, $\mathfrak{h}(z)=\log h(z)$ and $\nabla\rho(z)=\Big(\rho_{1}(z),\ldots,\rho_{n}(z)\Big),$ $ \ov{\nabla} \rho(z)=\Big(\rho_{\ov{1}}(z),\ldots,\rho_{\ov{n}}(z)\Big)$ and $\nabla\rho(z)^{t}$ denote the transpose of $\nabla\rho(z)$. Finally, for two $C^\infty$-smooth functions $f_1\text{ and }f_2$ on $\Om$, we write $f_1=O\big(f_2\big)$ if there exists a positive constant $C$, depending only on $\Om$, such that 
\[\big|f_1\big|\leq C\big|f_2\big|\]
on $\Om$. In this case, for each $z\in\Om$, we write $f_1(z)=O\Big(f_2(z)\Big)$. The following lemmas are the key to proof Theorem \ref{geo-spi-l1}. 
\begin{lemma}\label{eog0}
    For $a,b,j=1,\ldots,n$, we have 
    \begin{enumerate}
    \item [(a)] \begin{equation*}
            \mathfrak{h}_{b}=  \left\{
\begin{array}{ll}
     O\Big(\log\big|\rho\big|\Big)\;&\text{ for }\;  n= 1 \\[4mm]
     O(1)\;&\text{ for }\;  n\geq 2.
       \end{array} 
\right.
\end{equation*}
        \item [(b)]\begin{equation*}
            \mathfrak{h}_{b\ov{j}}=  \left\{
\begin{array}{ll}
   O\Big(\big(\log
    \big|\rho\big|\big)^2\Big)+O\Big(\rho^{-1}\Big)\;&\text{ for }\; n=1 \\[4mm]
     O\Big(\log\big|\rho\big|\Big)\;&\text{ for }\;  n= 2 \\[4mm]
     O(1)\;&\text{ for }\;  n\geq 3.
       \end{array} 
\right.
        \end{equation*}
        \item [(c)]\begin{equation*}
            \mathfrak{h}_{ab\ov{j}}=  \left\{
\begin{array}{ll}
     O\Big(\rho^{-2}\Big)+ O\Big(\big(\log
    \big|\rho\big|\big)^3\Big)+O\Big(\rho^{-1}\log
    \big|\rho\big|\Big) \;&\text{ for }\; n=1 \\[4mm]
     O\Big(\rho^{-1}\Big)+ O\Big(\log\big|\rho\big|\Big)\;&\text{ for }\; n= 2 \\[4mm]
     O\Big(\log\big|\rho\big|\Big) \;&\text{ for }\; n= 3 \\[4mm]
     O(1) \;&\text{ for }\; n\geq 4.
       \end{array} 
\right.
  \end{equation*}
    \end{enumerate}
\end{lemma}
\begin{proof} 
This result follows from the bare-hand computations on these terms $\mathfrak{h}_b, \mathfrak{h}_{b\ov{j}},$ and $\mathfrak{h}_{ab\ov{j}}$. Indeed, we have 
\begin{equation}\label{h_b}
    \mathfrak{h}_b=h_b h^{-1}.
\end{equation}
Observe that $h_b=O\Big(\big(\rho^n\log\big|\rho\big|\big)_b\Big)$, where
\begin{equation}\label{a1}
   \big(\rho^n\log\big|\rho\big|\big)_b= \left\{
\begin{array}{ll}
    O\Big(\log\big|\rho\big|\Big)\;&\text{ for }\;  n= 1 \\[4mm]
   O(1)\;&\text{ for }\;  n\geq 2.
       \end{array} 
\right.
\end{equation}
This completes the proof of (a). Next, from (\ref{h_b}),
\begin{equation}\label{h_bj-h2}
    \mathfrak{h}_{b\ov{j}}h^2=h_{b\ov{j}}h-h_b h_{\ov{j}}.
\end{equation}
So, by computing $h_{b\ov{j}}$, we get
\begin{equation}\label{h_bj}
    h_{b\ov{j}}=\Phi_{b\ov{j}}+\Psi_{b\ov{j}}\rho^{n}\log\big|\rho\big|+\Psi_{b}\big(\rho^{n}\log\big|\rho\big|\big)_{\ov{j}}+\Psi_{j}\big(\rho^{n}\log\big|\rho\big|\big)_{b}+\Psi\big(\rho^{n}\log\big|\rho\big|\big)_{b\ov{j}}.
\end{equation}
Hence, $ h_{b\ov{j}}=O\Big(\big(\rho^{n}\log\big|\rho\big|\big)_{b\ov{j}}\Big),$ where
\begin{equation}\label{a2}
  \big(\rho^{n}\log\big|\rho\big|\big)_{b\ov{j}}=\left\{
\begin{array}{ll}
    O\Big(\log
    \big|\rho\big|\Big)+O\Big(\rho^{-1}\Big)\;&\text{ for }\;  n= 1 \\[4mm]
      O\Big(\log
    \big|\rho\big|\Big)\;&\text{ for }\;  n= 2 \\[4mm]
 O(1)\;&\text{ for }\;  n\geq 3.
       \end{array} 
\right.
\end{equation}
Then, using (\ref{a1}), (\ref{h_bj}) and (\ref{a2}) in (\ref{h_bj-h2}), we are done. Similarly, the proof of $\mathfrak{h}_{ab\ov{j}}$ follows.
\end{proof}
      \begin{lemma}\label{eog1}
        Let $\Om\subset\mf{C}^n\text{ where }n\geq 2,$ be a $C^\infty$-smoothly bounded strongly pseudoconvex domain. Then,
        \begin{enumerate}
            \item [(a)]For each $j=1,\ldots,n$,
            \begin{equation*}
          \dfrac{\left[L_{g}^{-1} \cdot\big(\nabla\rho\big)^{t}\right]_{j}}{\rho^{2}}=O\big(1\big),
        \end{equation*}
        \item [(b)]    
        \begin{equation*}
        \dfrac{\overline{\nabla} \rho\cdot L_{g}^{-1}\cdot\big(\nabla \rho\big)^{t}}{\rho^{2}}-\frac{1}{n}=
     O\Big(\rho\log\big|\rho\big|\Big).
        \end{equation*}
   \end{enumerate}
  \end{lemma}
 
       \begin{proof}
 Using (\ref{F}), we obtain \begin{equation*}
    g_{j\ov{k}}=n\mathfrak{g}_{j\ov{k}}+\mathfrak{h}_{j\ov{k}},
\end{equation*}
where 
\begin{equation}\label{g2}
    \mathfrak{g}_{j\ov{k}}=\dfrac{\rho_{j\overline{k}}}{-\rho}+\dfrac{\rho_{j}\rho_{\overline{k}}}{\rho^{2}}.
\end{equation}
Then,
   \begin{equation*}
         L_g=n\left(I+\dfrac{1}{n}L_{\mathfrak{h}}\cdot L_{\mathfrak{g}}^{-1}\right)\cdot L_{\mathfrak{g}}.
   \end{equation*}
Hence,
   \begin{equation}\label{g3}
       L_{g}^{-1}=\dfrac{1}{n}L_{\mathfrak{g}}^{-1}-\dfrac{L_{\mathfrak{g}}^{-1}}{n^{2}}\cdot\left(I+\dfrac{1}{n}L_{\mathfrak{h}}\cdot L_{\mathfrak{g}}^{-1}\right)^{-1}L_{\mathfrak{h}}\cdot L_{\mathfrak{g}}^{-1}.
   \end{equation}
From (\ref{g2}),
   \begin{equation}\label{g4}
L_{\mathfrak{g}}^{-1}=\big|\rho\big|\left(L_{\rho}^{-1}-\dfrac{L_{\rho}^{-1}}{\big|\rho\big|+\mathcal{Q}}\cdot\big(\nabla\rho\big)^{t}\cdot\ov{\nabla}\rho\cdot L_{\rho}^{-1}
 \right),
   \end{equation}
   where $\mathcal{Q}=\overline{\nabla} \rho\cdot L_{\rho}^{-1}\cdot\big(\nabla \rho\big)^{t}$. This implies that
   \begin{equation}\label{g001}
       \dfrac{\big[L_{\mathfrak{g}}^{-1}\cdot\nabla\rho^{t}\big]_j}{\rho^{2}}=\frac{\left[L_{\rho}^{-1}\cdot\big(\nabla\rho\big)^{t}\right]_j}{\big|\rho\big|+\mathcal{Q}},
   \end{equation}
   and 
     \begin{equation}\label{g002}
         \dfrac{\ov{\nabla}\rho\cdot L_{\mathfrak{g}}^{-1}\cdot\big(\nabla\rho\big)^{t}}{\rho^{2}}=\frac{\ov{\nabla}\rho\cdot L_{\rho}^{-1}\cdot\big(\nabla\rho\big)^{t}}{\big|\rho\big|+\mathcal{Q}}.
     \end{equation}
 From (\ref{g3}), it is enough to examine $ \dfrac{1}{\rho^2}\left[L_{\mathfrak{g}}^{-1}\cdot\left(I+\dfrac{1}{n}L_{\mathfrak{h}}\cdot L_{\mathfrak{g}}^{-1}\right)^{-1}\cdot L_{\mathfrak{h}}\cdot L_{\mathfrak{g}}^{-1}\right]$. So, we proceed by considering
   \begin{multline*}
       \frac{1}{\rho^2}\left[L_{\mathfrak{g}}^{-1}\cdot\left(I+\dfrac{1}{n}L_{\mathfrak{h}}\cdot L_{\mathfrak{g}}^{-1}\right)^{-1}\cdot L_{\mathfrak{h}}\cdot L_{\mathfrak{g}}^{-1}\cdot\big(\nabla\rho\big)^{t}\right]_j\\
       =\sum_{s,l,k=1}^{n}\mathfrak{g}^{j\overline{l}}\left(I+\dfrac{1}{n}L_{\mathfrak{h}}\cdot L_{\mathfrak{g}}^{-1}\right)^{l\overline{s}}\mathfrak{h}_{s\ov{k}}\dfrac{\left[L_{\mathfrak{g}}^{-1}\cdot\big(\nabla\rho\big)^{t}\right]_{k}}{\rho^{2}}\\
       =\frac{1}{\big|\rho\big|+\mathcal{Q}}\sum_{s,l,k=1}^{n}\mathfrak{g}^{j\overline{l}}\left(I+\dfrac{1}{n}L_{\mathfrak{h}}\cdot L_{\mathfrak{g}}^{-1}\right)^{l\overline{s}}\mathfrak{h}_{s\ov{k}}\left[L_{\rho}^{-1}\cdot\big(\nabla\rho\big)^{t}\right]_k\text{ (by (\ref{g001}))}\\
       =\frac{1}{\big|\rho\big|+\mathcal{Q}}\sum_{s,l,k=1}^{n}\mathfrak{g}^{j\overline{l}}\mathfrak{h}_{s\ov{k}}\left(I+\dfrac{1}{n}L_{\mathfrak{h}}\cdot L_{\mathfrak{g}}^{-1}\right)^{l\overline{s}}\left[L_{\rho}^{-1}\cdot\big(\nabla\rho\big)^{t}\right]_k.
   \end{multline*}
Now, from (\ref{g4}), for fixed $z\in\Om$ and for any $X\in\mf{C}^n\setminus\{0\}$, 
 \begin{equation}\label{g5}
     \frac{\overline{X}^{t}\cdot L_{\mathfrak{g}}(z)^{-1}\cdot X}{|X|^2}= O\Big(\rho(z)\Big).
 \end{equation}
This implies that $\mathfrak{g}^{j\ov{k}}=O\big(\rho\big)$ for all $j,k=1,\ldots,n$. Then, by Lemma \ref{eog0} (b), we have
\begin{equation}\label{g6}
    \mathfrak{g}^{j\overline{l}}\mathfrak{h}_{s\ov{k}}  =  \left\{
\begin{array}{ll}
   O\Big(\rho\big(\log
    \big|\rho\big|\big)^2\Big)+O\big(1\big)\;&\text{ for }\; n=1 \\[4mm]
     O\Big(\rho\log\big|\rho\big|\Big)\;&\text{ for }\;  n= 2 \\[4mm]
     O\big(\rho\big)\;&\text{ for }\;  n\geq 3.
       \end{array} 
\right.
\end{equation}
It follows that the entries of $\big(L_{\mathfrak{h}}\cdot L_{\mathfrak{g}}^{-1}\big)(z)$ approach zero as $z\to \pa\Om$ for $n\geq 2$. Consequently, all coefficients of $\left(I+\dfrac{1}{n}L_{\mathfrak{h}}\cdot L_{\mathfrak{g}}^{-1}\right)^{-1}$ remain bounded on $\Om$.
%For $n=1$, the entries of $\left(I+\dfrac{1}{n}L_{\mathfrak{h}}\cdot L_{\mathfrak{g}}^{-1}\right)^{-1}$ are also bounded on $\Om$ since $\Big|\det\Big(I+\dfrac{1}{n}L_{\mathfrak{h}}\cdot L_{\mathfrak{g}}^{-1}\Big)\Big|$ is bounded below. To see why, suppose the contrary: that there exists a sequence $\{z_m\}\subset\Om$ such that $\Big|\det\Big(I+\dfrac{1}{n}L_{\mathfrak{h}}\cdot L_{\mathfrak{g}}^{-1}\Big)(z_m)\Big|\to 0$ as $m\to\infty$. This would imply that at least one of the eigenvalues $\la_j(z_m)~(1\leq j\leq n)$ of $\dfrac{1}{n}L_{\mathfrak{h}}\cdot L_{\mathfrak{g}}^{-1}(z_m)$, satisfies $\la_j(z_m)\to -1$ as $m\to\infty$. \textcolor{red}{Therefore, $|\det\big(L_{\mathfrak{h}}\cdot L_{\mathfrak{g}}^{-1}\big)(z_m)|\to\infty$ as $m\to\infty$} which contradicts the fact that $\det\big(L_{\mathfrak{h}}\cdot L_{\mathfrak{g}}^{-1}\big)$ is bounded in $\Om$, as follows from (\ref{g6}).
Hence,
\begin{equation}\label{g8}
      \frac{1}{\rho^2}\left[L_{\mathfrak{g}}^{-1}\cdot\left(I+\dfrac{1}{n}L_{\mathfrak{h}}\cdot L_{\mathfrak{g}}^{-1}\right)^{-1}\cdot L_{\mathfrak{h}}\cdot L_{\mathfrak{g}}^{-1}\cdot\big(\nabla\rho\big)^{t}\right]_j=  
     O\Big(\rho\log\big|\rho\big|\Big)\;\text{ for }\;  n\geq 2.
\end{equation}
Thus, the proof of (a) is complete. 
To prove (b), one can observe that 
 \begin{equation}\label{g9}
       \frac{1}{n}\dfrac{\ov{\nabla}\rho\cdot L_{\mathfrak{g}}^{-1}\cdot\big(\nabla\rho\big)^{t}}{\rho^{2}}-\frac{1}{n}=\frac{\mathcal{Q}}{n\big(\big|\rho\big|+\mathcal{Q}\big)}-\frac{\mathcal{Q}}{n\mathcal{Q}}=-\frac{\big|\rho\big|}{n(\big|\rho\big|+\mathcal{Q})}.
   \end{equation}
 Therefore, the proof (b) follows from (\ref{g8}) and (\ref{g9}).
   \end{proof}
    \begin{lemma}\label{eog2}
      Let $c=(c_1,\ldots,c_n)$ be a geodesic of the Szeg\"o metric $ds_{s_{\Om}}^2$. Then
        \begin{align*}
          (\rho\circ c)''=&-2\Re\sum_{a,b,j=1}^{n}\left[\left(\mathfrak{h}_{ab\overline{j}}-\dfrac{n}{\rho}\rho_{ab\overline{j}}\right)\left[L_{g}^{-1}\cdot \big(\nabla\rho\big)^{t}\right]_{j}(c)\cdot c_{a}'c_{b}'\right]\\
           &-\dfrac{4}{\rho(c)}\Re\left(c'\cdot \big(L_{\mathfrak{h}}\cdot L_{g}^{-1}\cdot\big(\nabla\rho\big)^{t}\big)(c)\right)\left[\nabla\rho(c)\cdot {\big(c'\big)}^{t}\right]\\
           &+2\left(1-\dfrac{n}{\rho^{2}}\overline{\nabla} \rho\cdot L_{g}^{-1}\cdot\big(\nabla \rho\big)^{t}\right)(c)\Re\left(\sum_{a,b=1}^{n}\rho_{ab}(c)c_{a}'c_{b}'\right)\\
           &+\dfrac{4}{\rho( c)}\Re\big(\nabla\rho( c)\cdot c'^{t}\big)^{2}+2c' \cdot L_{\rho}(c)\cdot\overline{{c'}^{t}}.
        \end{align*}
    \end{lemma}
\begin{proof}
The proof follows exactly the same lines of reasoning used in the proof of \cite[Lemma 3.4]{hg83}.
\end{proof}

  \begin{proof}[\textbf{Proof of Theorem \ref{geo-spi-l1}}]
       Suppose to the contrary, there exists a sequence $(c_{i})_{i\geq 1}$ of geodesics that satisfies the following:
\begin{itemize}
    \item[(i)] There exists a point $a_{0}\in \pa\Om$ such that $a_{i}=c_{i}(0)$ converges to $a_{0}$ as $i\to\infty$.
    \item[(ii)] The unit vectors $v_{i}=\dfrac{c_{i}'(0)}{|c_{i}'(0)|}$ converge to a unit vector $v_{0}$.
    \item[(iii)] We have $(\rho\circ c_{i})'(0)=0$ and $b_{i}=\dfrac{(\rho\circ c_{i})''(0)}{|c_{i}'(0)|^{2}}\leq 0$ for each $i$.
\end{itemize}
Using Lemma \ref{eog2}, we have
 \begin{multline}\label{LHS}
         b_i-4\frac{\Re\big(\nabla\rho(a_i)\cdot {v_i}^{t}\big)^{2}}{\rho(a_i)}-2v_i \cdot L_{\rho}(a_i)\cdot\overline{{v_{i}}^{t}}\\
         =-2\Re\sum_{a,b,j=1}^{n}\left[\left(\mathfrak{h}_{ab\overline{j}}-\dfrac{n}{\rho}\rho_{ab\overline{j}}\right)\left[L_{g}^{-1}\cdot (\nabla\rho)^{t}\right]_{j}(a_i)\cdot (v_{i})_{a}(v_{i})_{b}\right]\\
           -\dfrac{4}{\rho(a_i)}\Re\left(v_i\cdot (L_{\mathfrak{h}}\cdot L_{g}^{-1}\cdot(\nabla\rho)^{t})(a_i)\right)\left[\nabla\rho(a_i)\cdot {v_i}^{t}\right]\\
           +2\left(1-\dfrac{n}{\rho^{2}}\overline{\nabla}\rho\cdot L_{g}^{-1}\cdot(\nabla\rho)^{t}\right)(a_i)\Re\left(\sum_{a,b=1}^{n}\rho_{ab}(a_i)(v_{i})_{a}(v_{i})_{b}\right)\\
           =A_i+B_i+C_i,
        \end{multline}
 say. Before
proceeding into the examination of these terms, from (ii), we observe that $\big(\nabla\rho(a_i)\cdot {v_i}^{t}\big)_i$ is a sequence of imaginary numbers. Thus, by (iii), we have
\begin{equation}\label{b-0}
    \lim_{i\to\infty}\Bigg( b_i-4\frac{\Re\big(\nabla\rho(a_i)\cdot {v_i}^{t}\big)^{2}}{\rho(a_i)}-2v_i \cdot L_{\rho}(a_i)\cdot\overline{{v_{i}}^{t}}\Bigg)< 0.
\end{equation}
In what follows, we will derive a contradiction to (\ref{b-0}).
 
 \textbf{The term $A_i$:} By Lemma \ref{eog0} (c) and Lemma \ref{eog1} (a), 
        \begin{equation}\label{A_i}
            A_i=  \left\{
\begin{array}{ll}
     O\Big(\rho(a_i)^{2}\Big)\;&\text{ for }\; n= 2 \\[4mm]
     O\Big(\rho(a_i)^{2}\log\big|\rho(a_i)\big|\Big) \;&\text{ for }\; n= 3 \\[4mm]
     O\Big(\rho(a_i)^{2}\Big) \;&\text{ for }\; n\geq 4.
       \end{array} 
\right.
        \end{equation}
        Hence, $A_i\to 0$ as $i\to\infty$.
 
  \textbf{The term $B_i$:} By Lemma \ref{eog0} (b) and Lemma \ref{eog1} (a), 
        \begin{equation}\label{B_i}
           B_i= \left\{
\begin{array}{ll}
     O\Big(\rho(a_i)\log\big|\rho(a_i)\big|\Big)\;&\text{ for }\;  n= 2 \\[4mm]
     O\Big(\rho(a_i)\Big)\;&\text{ for }\;  n\geq 3.
       \end{array} 
\right.
        \end{equation} 
Thus, $B_i\to 0$ as $i\to \infty$ for $n\geq 2$.
 
 \textbf{The term $C_i$:} By Lemma \ref{eog1} (b), 
 \begin{equation}\label{C_i}
     C_i= O\Big(\rho(a_i)\log\big|\rho(a_i)\big|\Big)\;\;\;\text{ for }\;  n\geq 2.
 \end{equation}
This implies $C_i\to 0$ as $i\to \infty$.
  Therefore, using (\ref{A_i}), (\ref{B_i}) and (\ref{C_i}) in (\ref{LHS}), we obtain the contradiction to (\ref{b-0}).
    \end{proof}

\subsection*{Concluding remarks.} The techniques used to establish the existence of geodesic spirals in this article fail in the case of $C^\infty$-smoothly bounded, non-simply connected planar domains. This is due to the lack of information regarding the boundedness of
\[\Bigg[\Big(I+\dfrac{1}{n}L_{\mathfrak{h}}\cdot L_{\mathfrak{g}}^{-1}\Big)^{-1}\Bigg]_j,\;j=1,\ldots, n.\]
Even if the boundedness of these terms is known—implying that the right-hand sides of (a) and (b) in Lemma \ref{eog1} is $O(1)$---which does not ensure the convergence of $A_i, B_i, C_i$ to zero in (\ref{LHS}) in the proof of Theorem \ref{geo-spi-l1} because of (b) and (c) of Lemma \ref{eog0}. Nevertheless, the scaling method remains applicable in this scenario; see \cite{bb24}.

The qualitative behaviour of geodesics for the Szeg\"o metric on an annulus remains unknown, in contrast to the Bergman metric; see \cite[Theorem 4.2]{hg83}. However, based on Example \ref{Anuulus}, we expect the geodesics of the Szeg\"o metric on an annulus to exhibit behaviour similar to those for the Bergman metric.

Does there exist a strongly pseudoconvex $C^\infty$-smoothly bounded domain $\Om$ such that $\big(\Om, ds_{s_\Om}^2\big)$ is not a Hadamard-Cartan manifold which possesses neither closed nor geodesic spirals? The existence of such kinds of domains holds for the Bergman metric; see \cite[Theorem 5.1]{hg83}.

On another note, it is natural to ask whether a geodesic $c(t)$ for the Szeg\"o metric, which does not remain within a compact subset of a $C^\infty$-smoothly bounded strongly pseudoconvex domain $\Om\subset\mf{C}^n$ for all $t\geq 0$, will eventually hit the boundary $\pa \Om$? 
%----------------------------------------------------------------------------------------------------------------------
\bibliographystyle{halpha}
\bibliography{Derived}

\begin{thebibliography}{BdMS76}
\expandafter\ifx\csname url\endcsname\relax
  \def\url#1{\texttt{#1}}\fi
\expandafter\ifx\csname doi\endcsname\relax
  \def\doi#1{\burlalt{doi:#1}{http://dx.doi.org/#1}}\fi
\expandafter\ifx\csname urlprefix\endcsname\relax\def\urlprefix{URL }\fi
\expandafter\ifx\csname href\endcsname\relax
  \def\href#1#2{#2}\fi
\expandafter\ifx\csname burlalt\endcsname\relax
  \def\burlalt#1#2{\href{#2}{#1}}\fi

\bibitem[BB24]{bb24}
A.~Bhatnagar and D.~Borah.
\newblock Some remarks on the carath\'eodory and szeg\"o metrics on planar domains.
\newblock {\em arxiv}, page~28, 2024.
\newblock \doi{arXiv:2410.20955v2}.

\bibitem[BdMS76]{bs76}
L.~Boutet~de Monvel and J.~Sj\"ostrand.
\newblock Sur la singularit\'e{} des noyaux de {B}ergman et de {S}zeg\"o.
\newblock In {\em Journ\'ees: \'Equations aux {D}\'eriv\'ees {P}artielles de {R}ennes (1975)}, volume No. 34--35 of {\em Ast\'erisque}, pages 123--164. Soc. Math. France, Paris, 1976.

\bibitem[BL14]{bl14}
David Barrett and Lina Lee.
\newblock On the {S}zeg\"o metric.
\newblock {\em J. Geom. Anal.}, 24(1):104--117, 2014.
\newblock \doi{10.1007/s12220-012-9329-x}.

\bibitem[Fef74]{fc74}
C.~Fefferman.
\newblock The {B}ergman kernel and biholomorphic mappings of pseudoconvex domains.
\newblock {\em Invent. Math.}, 26:1--65, 1974.
\newblock \doi{10.1007/BF01406845}.

\bibitem[Hen73]{h73}
G.~M. Henkin.
\newblock An analytic polyhedron is not holomorphically equivalent to a strictly pseudoconvex domain.
\newblock {\em Dokl. Akad. Nauk SSSR}, 210:1026--1029, 1973.

\bibitem[Her83]{hg83}
G.~Herbort.
\newblock On the geodesics of the {B}ergman metric.
\newblock {\em Math. Ann.}, 264(1):39--51, 1983.
\newblock \doi{10.1007/BF01458049}.

\bibitem[Kra19]{k19}
S.~G. Krantz.
\newblock The {F}efferman-{S}zeg\"o metric and applications.
\newblock {\em Complex Var. Elliptic Equ.}, 64(6):965--978, 2019.
\newblock \doi{10.1080/17476933.2018.1489800}.

\bibitem[KW21]{k21}
S.~G. Krantz and P.~M. W\'ojcicki.
\newblock On an invariant distance induced by the {S}zeg\"o kernel.
\newblock {\em Complex Anal. Synerg.}, 7(3):Paper No. 24, 9, 2021.
\newblock \doi{10.1007/s40627-021-00085-w}.

\bibitem[MP93]{jp93}
Jarnicki M. and P.~Pflug.
\newblock {\em Invariant distances and metrics in complex analysis}, volume~9 of {\em De Gruyter Expositions in Mathematics}.
\newblock Walter de Gruyter \& Co., Berlin, 1993.
\newblock \doi{10.1515/9783110870312}.

\bibitem[Zwo10]{z10}
W.~Zwonek.
\newblock Asymptotic behavior of the sectional curvature of the {B}ergman metric for annuli.
\newblock {\em Ann. Polon. Math.}, 98(3):291--299, 2010.
\newblock \doi{10.4064/ap98-3-8}.

\end{thebibliography}

\end{document}